\documentclass[11pt]{article}
\usepackage{latexsym,subfigure}
\usepackage{comment}
\usepackage{amsmath,amsthm,amsfonts,amssymb,graphicx,epsfig,latexsym,color}
\usepackage{epsf,enumerate}
\newcommand{\orsubset}{\overset{\circ}{\subset}}
\title{Bounded cohomology with coefficients in uniformly convex Banach spaces}
\author{Mladen Bestvina, Ken Bromberg and Koji Fujiwara\thanks{The
    first two authors gratefully acknowledge the support by the National
    Science Foundation. The third author is supported in part by
Grant-in-Aid for Scientific Research (No. 23244005)}} \date{\today}
\def\Bbb{\mathbb}

\newtheorem{thm}{Theorem}[section]
\newtheorem{lemma}[thm]{Lemma}
\newtheorem{cor}[thm]{Corollary}
\newtheorem{prop}[thm]{Proposition}
{}
{}

\theoremstyle{remark}
\newtheorem{example}[thm]{Example}

\newtheorem{definition}[thm]{Definition}
\newtheorem{remark}[thm]{Remark}
\newtheorem{examples}[thm]{Examples}
\newtheorem*{definition*}{Definition}
\newtheorem*{remark*}{Remark}

\def\R{{\mathbb R}}

\renewcommand{\>}{\rangle}

\begin{document}

\maketitle

\begin{abstract}
We show that for acylindrically hyperbolic groups $\Gamma$ (with no
nontrivial finite normal subgroups) and arbitrary unitary representation
$\rho$ of $\Gamma$ in a (nonzero) uniformly convex Banach space the
vector space $H^2_b(\Gamma;\rho)$ is infinite dimensional. 
The result was known for the regular representations 
on $\ell^p(\Gamma)$ with $1<p<\infty$ by a different 
argument.  But our result is new even for a non-abelian free group
in this great generality for representations, and also the case for acylindrically hyperbolic
groups follows as an application. 
\end{abstract}

\section{Introduction}
\subsection{Quasi-cocycle and quasi-action}
Let $G$ be a group and $E$ a normed vector space (usually complete,
either over $\R$ or over $\mathbb C$).  The linear or rotational part
of an isometric $G$-action on $E$ determines a representation
$\rho:G\to O(E)$ where $O(E)$ is the group of norm-preserving linear
isomorphisms $E\to E$.  We will refer to $\rho$ as a {\it unitary
  representation}.  We will usually write $\rho(g)x$ as $g(x)$ or
$gx$.

The translational part of the $G$-action is a {\it cocycle} (with
respect to $\rho$). Namely the translational part is a function
$F:G\to E$ that satisfies
\begin{equation}\label{cocycle}
F(gg')=F(g)+gF(g')
\end{equation}
for all $g,g'\in G$. Going in the other direction, if $\rho$ is a
unitary representation and $F$ a cocycle then the map $g\mapsto
(x\mapsto \rho(g)x+F(g))$ determines an isometric $G$-action on
$E$. Note that $F(g^{-1})=-g^{-1}F(g)$.

For an isometric quasi-action of $G$ on $E$ the linear part will still
be a unitary representation. However, the translational part $F$ will
become a {\it quasi-cocycle} and will only satisfy \eqref{cocycle} up
to a uniformly bounded error so that
\begin{equation}\label{quasi-cocycle}
\Delta(F):=\sup_{g,g'\in G}|F(gg')-F(g)-gF(g')|<\infty.
\end{equation}
The quantity $\Delta(F)$ is the {\it defect} of the quasi-cocycle.

A basic question is if there are quasi-actions that are not boundedly
close to an actual action. Such a quasi-action is {\it
  essential}. Since quasi-actions determine unitary representations a
more refined question is if there are essential quasi-actions for a
given unitary representation. 

The above discussion is perhaps more familiar in its algebraic form
where it can be rephrased in terms of bounded cohomology. A
quasi-cocycle $F$ can be viewed as 1-cochain in the group cohomology
twisted by the representation $\rho$. Condition \eqref{quasi-cocycle},
is equivalent to the coboundary $\delta F$ being a bounded 2-cocycle
and will therefore determine a cohomology class in $H^2_b(G;\rho)$,
the second bounded cohomology group. Now this cocycle will clearly be
trivial in the regular second cohomology group $H^2(G;\rho)$ as it is
the coboundary of a 1-cochain.  If the cochain $F$ is a bounded
distance from a cocycle then $\delta F$ will also be trivial in
$H^2_b(G;\rho)$ so we are interested in the kernel of the map
$$H^2_b(G;\rho)\to H^2(G;\rho)$$ from bounded cohomology to regular
cohomology. In particular this kernel is the vector space $QC(G;\rho)$
of all quasi-cocycles modulo the subspace generated by bounded
functions and cocycles. We denote this quotient space
$\widetilde{QC}(G;\rho)$. This is the vector space of {\it essential}
quasi-cocycles and it is the main object of study of this paper.

For the trivial representation on $\R$ a cocycle is just a
homomorphism to $\R$ and a quasi-cocycle is usually called a
quasi-morphism. When $G= F_2$, the free group on two generators,
Brooks \cite{brooks} gave a combinatorial construction of an infinite
dimensional family of essential quasi-morphisms. 

\subsection{Uniformly convex Banach space and main result}

Following the work of Brooks, there is a long history of
generalizations of this construction to other groups. Initially, the
work focused on the trivial representation. See
\cite{epstein-fujiwara,bestvina-fujiwara,rank1}. This was followed by generalizations to the
same groups $G$ but with coefficients in the regular representation
$\ell^p(G)$, $1\leq p<\infty$. See \cite{ursula.rank1,ursula2}.

In this paper
we will extend this work to unitary representations in {\it uniformly
  convex} Banach spaces. Note that this essentially includes the
previous cases since $\ell^p(G)$ is uniformly convex when $1<p<
\infty$.

If one is a bit more
careful about how the counting is done then Brooks construction of
quasi-morphisms can also be used to produce quasi-cocycles. In Brooks'
original  work (i.e., for trivial representations) it is easy to see that the quasi-morphisms are
essential. Here we will have to work harder to get the following
result.

\begin{thm}[Theorem \ref{main:free}]\label{main:free2}
Let $\rho$ be a unitary representation of $F_2$ on a uniformly convex
Banach space $E\neq 0$.
Then
$\dim\widetilde{QC}(F_2;\rho)=\infty$. 
\end{thm}

To show $\widetilde{QC}(F_2;\rho)$ is non-trivial is already hard.
We will argue for a certain Brooks' quasi-cocycle $H$ into a Banach space $E$, 
there exists a sequence of elements in $F_2$ on which $H$ is unbounded.
For that we use that $E$ is uniformly convex in an essential way (Lemma \ref{unbounded}).
We also show those quasi-cocycle are not at bounded distance
from any cocycle using that $E$ is reflexive (using Lemma \ref{bounded cocycle 2}). 
Those two steps are the novel part of the paper. It seems that the uniform convexity 
is nearly a necessary assumption for the conclusion. See the examples
at the end of this section.

Recently Osin \cite{osin2} (see also \cite{DGO}) has
identified the class of {\em acylindrically hyperbolic groups} and
this seems to be the most general context where the Brooks'
construction can be applied.  Osin has shown that acylindrically
hyperbolic groups contain {\em hyperbolically embedded} copies of
$F_2$ and then applying work of Hull-Osin (\cite{hull.osin}) we have
the the following corollary to Theorem \ref{main:free}.
See Section \ref{s:hyp-embed} for the proof. 

\begin{cor}\label{main:acylind}
Let $\rho$ be a unitary representation of an acylindrically hyperbolic
group $G$ on a uniformly convex Banach space $E \neq 0$ and assume
that the maximal finite normal subgroup has a non-zero fixed
vector. Then $\dim \widetilde{QC}(G;\rho) = \infty$.
\end{cor}
 
A wide variety of groups are acylindrically hyperbolic. In particular
our results apply to the following examples. 
  To apply our result, in all examples assume $G$ has no nontrivial
  finite normal subgroups, or more generally that for the maximal
  finite normal subgroup $N$ (see \cite{DGO}) we have that $\rho(N)$
  fixes a nonzero vector in $E$.  
\begin{examples}[Acylindrically hyperbolic groups]
\begin{itemize}
\item $G$ is non-elementary word hyperbolic,
\item $G$ admits a non-elementary 
isometric action on a connected
  $\delta$-hyperbolic space such that at least
one element is hyperbolic and WPD,
\item $G=Mod(S)$, the mapping class group of a compact surface which
  is not virtually abelian,
\item $G=Out(F_n)$ for $n\geq 2$,
\item $G$ admits a non-elementary isometric action on a $CAT(0)$ space
  and at least one element is WPD and acts as a rank 1 isometry.
  (\cite{rank1} for $\rho=\R$ and \cite{ursula.rank1} for
  $\rho=\ell^p(G)$, both under the assumption that the CAT(0) space
  $X$ and the action of $G$ are proper and $G$ contains a
  rank-1 element).
\end{itemize}
%\end{main}
\end{examples}

\begin{remark}
Recall that a Banach space is {\it superreflexive} if it admits an equivalent
uniformly convex norm. It is observed in \cite[Proposition 2.3]{bfgm}
that if $\rho:G\to E$ is a unitary representation with $E$
superreflexive, then there is an equivalent uniformly convex norm with
respect to which $\rho$ is still unitary. Thus in Corollary
\ref{main:acylind} we may replace ``uniformly convex'' with
``superreflexive''. 
\end{remark}

\begin{remark}
There is also a more direct approach to going from Theorem
\ref{main:free} to our the main theorem. The key point is that any
group $G$ covered in the the main theorem acts on a quasi-tree such
that there is a free group $F \subset G$ that acts properly and
co-compactly on a tree isometrically embedded in the quasi-tree. This
is done using the {\em projection complex} of \cite{bbf}. Using this
one can apply the Brooks' construction to produce quasi-cocycles that
when restricted to the free group are exactly the quasi-cocycles of
Theorem \ref{main:free}. We carry this out in a separate paper
(\cite{bc}).
\end{remark}

\subsection{Known examples with certain Banach spaces}
Here are some known vanishing/non-vanishing examples in the literature.

\begin{itemize}
\item $E = \R$ and $\rho$ is trivial. In this case $H^2_b(G;\rho)$ is
  the usual bounded cohomology and quasi-cocycles are quasi-morphisms.
  As we said this case was known for various kinds of groups.

\item $E = \ell^p(G)$ and $\rho$ is the regular representation, see
  \cite{ursula.cat,ursula}. When $1<p< \infty$, $\ell^p(G)$ is
  uniformly convex and our theorem applies. When $p = 1$ or $p=\infty$
  then $\ell^p(G)$ is not uniformly, or even strictly,
  convex. However, for $p=1$ summation determines a $\rho$-invariant
  functional and one can produce a family of quasi-cocycles that when
  composed with the invariant functional are an infinite dimensional
  family of non-trivial quasi-morphisms in $\widetilde{QH}(G)$
  implying that $\dim \widetilde{QC}(G;\ell^1(G)) = \infty$.
\end{itemize}
On the other hand, 
\begin{itemize}
\item
When $p=\infty$ given any quasi-cocycle one can explicitly find a
cococyle a bounded distance away so $\widetilde{QC}(G;\ell^\infty(G))=0$ for any
group $G$.
\item
If $G$ is countable and  exact (e.g., $F_2$), then $H^2_b(G;\ell_0^\infty(G))=0$.
In particular, $\widetilde{QC}(G;\ell^\infty_0(G))=0$
(Example \ref{ex.ell.0}). Here $\ell^\infty_0(G)$ is the subspace of
$\ell^\infty(G)$ consisting of sequences which are asymptotically 0. 
\end{itemize}

\noindent
There are also examples where $G$ is not acylindrically hyperbolic but where
$\widetilde{QC}(G; \rho)$ is known to be non-zero for certain actions
of $G$ on $\ell^p$ spaces.

\begin{itemize}
\item If $G$ has a non-elementary action on a $CAT(0)$ cube complex
  then $\widetilde{QC}(G; \rho) \neq 0$ where $\rho$ is the
  representation of $G$ on the space of $\ell^p$-functions ($1 \leq p
  < \infty$) on a certain space where $G$ naturally acts
  (\cite{CFI}). Note that this class of groups is closed under
  products so it contains groups that aren't acylindrically
  hyperbolic.
\end{itemize}

\noindent
There are other examples where essentially nothing is known.

\begin{itemize}
\item $E = \ell^1_0(G) \subset \ell^1(G)$ is the space of $\ell^1$-functions on $G$ that sum to zero and $\rho$ is the regular representation. Unlike with $\ell^1(G)$, $\ell^1_0(G)$ has no $\rho$-invariant functionals.

\item $E = {\mathcal B}(\ell^2(G))$ the space of bounded linear maps
  of $\ell^2(G)$ to itself. This example was suggested to us by
  N. Monod as the non-commutative analogue to $\ell^\infty(G)$.
\end{itemize}

\section{Quasi-cocycles from trees}\label{s:free}

Fix $F_2= \langle a,b\rangle$ and choose a word $w\in F_2$. For simplicity we
will assume that $w$ is cyclically reduced.
% so that no two (oriented) copies of $w$ inside a word $g$ overlap. 
Let $E$ be a normed vector space and $\rho:G\to O(E)$ a linear
representation. Also choose a nonzero $e\in E$. We now set up some
notation that will be convenient for what we will do later.

Let $[g,h]$ be an oriented segment in the Cayley graph for $F_2$ with
generators $a$ and $b$. Then we write $[g,h] \orsubset [g', h']$ if $[g,h]$ is a
subsegment of $[g',h']$ and the orientations of the two segments
agree. We then define
$$w_+(g) = \{h \in G | [h, hw] \orsubset [1,g]\}$$
and
$$w_-(g) = \{h \in H | [h,hw] \orsubset [g,1]\}.$$

Now define a function $H=H_{w,e}:F_2\to E$ by
$$H(g)=\sum_{h \in w_+(g)}h(e)-\sum_{h\in w_-(g)}h(e)$$

In other words, to a translate $h\cdot w$ we assign $h(e)$ when
traversed in the positive direction, and $-h(e)$ when traversed in
negative direction. Note that it follows that $H(g^{-1})=-g^{-1}H(g)$.

\begin{prop}\label{free-Brooks}
The function $H$ constructed above is a quasi-cocycle.
\end{prop}

\begin{proof}
This is the standard Brooks argument. Consider the tripod spanned by
$1,g,gf$. Call the central point $p$. We will see that contributions
of copies of $w$ in the tripod that do not cross $p$ cancel out
leaving only a bounded number of terms.

If $h\cdot w\orsubset [1,p]$ then $h(e)$ enters with positive sign in
$H(g)$ and in $H(gf)$, so it cancels in the expression
$H(gf)-H(g)$. Likewise, if $h\cdot w\orsubset [p,1]$ then $-h(e)$
enters both $H(g)$ and $H(gf)$, so it again cancels.

If $h\cdot w\orsubset [p,g]$ then $h(e)$ is a summand in $H(g)$. Since
$h\cdot w\orsubset [gf,g]$ we also have $g^{-1}h\cdot w\orsubset
[f,1]$, so $-g^{-1}h(e)$ is a summand in $H(f)$, and thus we have
cancellation in $-H(g)-gH(f)$. There is similar cancellation if
$h\cdot w\orsubset [g,p]$.

If $h\cdot w\orsubset [p,gf]$ or $[gf,p]$ then similarly to the
previous paragraph there is cancellation in $H(gf)-gH(f)$.

After the above cancellations in the expression $H(gf)-H(g)-gH(f)$ the
only terms left are of the form $\pm h(e)$ where $h(w)$ is contained
in the tripod and contains $p$ in its interior. The number of such
terms is clearly (generously) bounded by $6|w|$ so we deduce that
$\Delta(H)\leq 6|w|\|e\|$.
\end{proof}
\begin{remark}
Note that if $h\cdot w$ does not overlap $w$ for any $1\not=h \in F_2$, 
then $\Delta(H) \le 6\|e\|$. More generally, for a given $w$, write
$w=a^nb$ as a word such that $|b|<|a|$ and $n>0$ is maximal.
Then,  $\Delta(H) \le 6(n+1)\|e\|$.

\end{remark}

\begin{example}\label{ab}
Suppose $w=ab$. Then $H(a^n)=H(b^n)=0$, while
$H((ab)^n)=(1+ab+(ab)^2+\cdots+(ab)^{n-1})e\in E$. If the operator
$1-ab:E\to E$ has a continuous inverse (i.e. if $1\in\mathbb C$ is not in the
spectrum of $ab$) then $H$ is uniformly bounded on the powers of $ab$
since $(1-ab)H((ab)^n)=e-(ab)^n(e)$ has bounded norm. For example,
this happens even for $E=\R^2$ when $\rho(ab)$ is a (proper) rotation.

On the other
hand, for the representation $\ell^p(F_2)$ with $1\leq p<\infty$ and
with $e\in\ell^p(F_2)$ defined by $e(1)=1$, $e(g)=0$ for $g\neq 1$,
the quasi-cocycle $H$ is unbounded on the powers of $ab$.
\end{example}

\section{Nontriviality of quasi-cocycles}\label{nontrivial}
In Brooks' original construction of quasi-morphisms $F_2=\<a,b\>\to\R$
it is easy to see
that the quasi-morphisms are nontrivial. Choosing $w$ to be a reduced
word not of the form $a^m$ or $b^m$ it is clear that $H(w^n)$ will
be unbounded while $H(a^n)$ and $H(b^n)$ will be
zero. By this last fact if $G$ is a homomorphism that is boundedly
close to $H$ then $G$ must be bounded on powers of $a$ and $b$ and
therefore $G(a) = G(b) =0$. Since any homomorphism is determined by
its behavior on the generators we have $G \equiv 0$ and the
nontriviality of $H$ follows.

When the Brooks construction is extended to quasi-cocycles it is no
longer clear that the quasi-cocycle is nontrivial. In particular if $H
= H_{w,e}$ it may be that $H(w^n)$ is bounded. See Examples \ref{ab} and
\ref{U2}. In fact if $1$ is not in
the spectrum of $\rho(w)$ then $H(w^n)$ will be bounded for all
choices of vectors $e$. Even if $1$ is in the spectrum, when $e$ is
chosen arbitrarily $H(w^n)$ may be bounded. To show that the Brooks
quasi-cocycles are unbounded we will need to restrict to the class of
{\it uniformly convex} Banach spaces and to look at a wider class of
words than powers of $w$.

We will also have to work harder to show that a cocycle $G$ that is
bounded on powers of the generators is bounded everywhere. In fact we
cannot do this in general but instead will show that in a reflexive
Banach space (which includes uniformly convex Banach spaces) either
the cocycle is bounded or the original representation, when restricted
to a non-abelian subgroup, has an eigenvector. In this latter case it
is easy to construct many nontrivial quasi-cocycles.

\subsection{Uniformly convex and reflexive Banach spaces}

We will use basic facts about Banach spaces. General references are
\cite{bourbaki, megginson}.
The following concept was introduced by Clarkson \cite{clarkson}.

\begin{definition}
A Banach space $E$ is {\it uniformly convex} if for every $\epsilon>0$
there is $\delta>0$ such that $x,y\in E$, $|x|\leq 1$, $|y|\leq 1$,
$|x-y|\geq\epsilon$ implies $|\frac {x+y}2|\leq 1-\delta$.
\end{definition}

The original definition in \cite{clarkson} replaces $|x|,|y|\leq 1$
above with equalities, but it is not hard to see that the two are
equivalent. 

\begin{prop}\label{UC basic}
\begin{enumerate}[(i)]
\item $\ell^p$ spaces are uniformly convex for $1<p<\infty$
  (\cite{clarkson}).
  $\ell^1$ and $\ell^{\infty}$ spaces are not uniformly convex and not
  reflexive. 
\item A uniformly convex Banach space is reflexive (the Milman-Pettis
  theorem).
\item If $E$ is uniformly convex, then for any $R>0$ there are
  $\epsilon>0$ and $\mu>0$ so that the following holds. If $|v|\le R$
  and $f:E\to\R$ is a functional of norm 1 with $f(v)=|v|$ and if $e$
  is a vector of norm $\geq 1/2$ with $f(e)\geq -\mu$ then $|v+e|\geq
  |v|+\epsilon$.
\end{enumerate}
\end{prop}

\begin{proof} We only prove
  (iii). Choose $\delta \in (0,1)$ so that $|x|,|y|\leq 1$, $|x-y|\geq
  \frac 1{2(R+1)}$ implies $|\frac{x+y}2|\leq 1-\delta$. Then choose
  $\epsilon,\mu>0$ so that $\epsilon<\frac 18$ and $\frac{\frac
    18-\frac {\mu}2}{\frac 18+\epsilon}>1-\delta$. 
  Suppose $f,v,e$ satisfy the assumptions but
  $|v+e|<R+\epsilon$. If $|v| \le 1/8$ then $|v + e| \ge |e| - |v| \ge
  1/4 \ge |v| + 1/8$ and we are done. So assume that $|v| >1/8$. Then
  for $x=\frac v{|v|+\epsilon}$, $y=\frac {v+e}{|v|+\epsilon}$ we have
  $|x|,|y|\leq 1$ and $|x-y|\ge \frac1{2(|v|+1)} \geq\frac 1{2(R+1)}$, so
  we must have $|\frac {x+y}2|\leq 1-\delta$. Thus
  $$1-\delta\geq
  \bigg|\frac{x+y}2\bigg|=\bigg|\frac{v+e/2}{|v|+\epsilon}\bigg|\geq
  \frac {|v|-\frac{\mu}2}{|v|+\epsilon}\geq\frac{\frac
    18-\frac{\mu}2}{\frac 18+\epsilon}$$ since
  $f(v+e/2)=|v|+f(e)/2\geq |v|-\frac{\mu}2$ and $|f|=1$. This
  contradicts the choice of $\mu,\epsilon$.
\end{proof}

\begin{lemma}\label{fixed points}
Let $\rho$ be a unitary representation of a group $F$ on a reflexive
Banach space $E$. If there is a linear functional $f$ and a vector $e
\in E$ such that the $F$-orbit of $e$ lies in the half space $\{f \ge
\mu\}$ with $\mu >0$ then there is an $F$-invariant vector $e'
\neq 0 \in E$ and a $F$-invariant functional $\phi$ with $\phi(e') \ge
\mu$.  If $e$ is $F$-invariant, then we can take $e'=e$.
\end{lemma}

\begin{proof}
Let $\Lambda$ be the convex hull of the $F$-orbit of $e$ in the weak
topology on $E$. Since $E$ is reflexive, $\Lambda$ is weakly
compact. The convex hull $\Lambda$ is also $F$-invariant so by the
Ryll-Nardzewski fixed point theorem it will contain an $F$-invariant
vector $e'$. Since $e' \in \Lambda$, $f(e') \ge \mu$ and therefore
$e' \neq 0$.

Since $e'$ is a functional on the reflexive Banach space $E^*$ and the $F$-orbit of $f$ will be contained in the half space $\{e' \ge \mu\}$ we similarly get a $F$-invariant vector $\phi \in E^*$ with $e'(\phi) = \phi(e') \ge \mu$.
\end{proof}

Note that if $E$ contains a nonzero vector that is $F$-invariant, then
the Hahn-Banach theorem supplies a functional that satisfies the
conditions of the lemma and so there is also a nonzero $F$-invariant
functional.

\subsection{Detecting unboundedness}

\begin{lemma}\label{unbounded}
Let $\rho$ be any unitary representation of $F_2 = \langle a,b\rangle$
into a uniformly convex Banach space $E$. Then one of the following
holds:
\begin{enumerate}[(i)]
\item for every $e\neq 0 \in E$ and any $1\neq w\in F_2$ not of the form
  $a^mb^n$ nor $b^ma^n$ the quasi-cocycle $H=H_{w,e}$ is unbounded on
  $F_2$, or
\item there is a free subgroup $F\subset F_2$ with $F\cong F_2$, a linear
  functional $g$, a vector $e$ and a $\mu>0$ such that the
  $F$-orbit of $e$ is contained in the half-space $\{g \le -\mu\}$.
  In particular, there is an $F$-invariant vector $e' \not=0$ in the half space.
\end{enumerate}
\end{lemma}

\begin{proof}
We first make some observations about words in $F_2$. Given a word $w$
as in (i) we can find buffer words $B$ and $B'$ of the form $a^\ell b^\ell$
or $b^\ell a^\ell$ and a subgroup $F = \langle a^m, b^m \rangle$ with $m
>>\ell, |w|$ such that if $w' = BwB'$ and $y_1, y_2, \dots, y_n \in F$ then in
the reduced word for the element $x=y_1 w' y_2 w' \cdots y_n w'$ there
is exactly one copy of $w$ for each $w'$ and no other copies of either
$w$ or $w^{-1}$. Note that the word $y_1 w' y_2 w' \cdots y_n w'$ may
not be reduced and in its reduced version there may be cancellations
in the $w'$. However, the buffer words will prevent these
cancellations from reaching $w$. The restrictions on $w$ ensure that
$w$ does not appear as a subword of some $y_i$. In particular,
$|H(w')|=|e|$ and $H(xyw')=H(x)+xH(yw')=H(x)+xyH(w')$ for any $y\in F$.

For simplicity, normalize so that $|e|=1$, so $|H(w')|=1$. Assume that (ii) doesn't
hold, and that $H$ is bounded on $F_2$. Let 
$F_w$ be the set of words of the
form 
$$y_1 w' y_2 w' \cdots y_n w' , (y_i  \in F)$$ and let $R= \underset{x\in
  F_w}{\sup} |H(x)|<\infty$. Let $\epsilon,\mu>0$ be as in Proposition
\ref{UC basic}(iii). Choose an $x\in F_w$ such that $|H(x)| > R -
\epsilon$. We will find a $y \in F$ such that $|H(xyw')| > R$ to
obtain a contradiction since $xyw' \in F_w$.

Let $\phi$ be a linear functional of norm 1 such that $\phi(H(x)) =
|H(x)|$. Let $\psi = \phi\circ x$. Since (ii) doesn't hold, there
exists a $y \in F$ with $\psi(yH(w')) > -\mu$. (We are applying the
negation of (ii) not to $e$ but to $H(w')$, which is in the
$F_2$-orbit of $e$, but it is easy to see that this follows from the
corresponding fact for $e$ by replacing $F$ with a conjugate.)
So, $\phi(xyH(w')) > -\mu$.
Then
by Proposition \ref{UC basic}(iii), $|H(xyw')| = |H(x) + xyH(w')| \ge
|H(x)| + \epsilon > R$, contradiction.

For an $F$-invariant vector in (ii), see the proof of Lemma \ref{fixed points}.
\end{proof}

We give an application of Lemma \ref{unbounded}.
\begin{example}\label{U2}
Choose an embedding $\rho:F_2\subset U(2)$ so that every nontrivial element is
conjugate to a matrix of the form
\[
\left(\begin{matrix}
e^{2\pi it}&0\\
0&e^{2\pi is}\end{matrix}\right)
\]
with $t,s,\frac ts$ all irrational. (Such representations can be
constructed by noting that they form the complement of
countably many proper subvarieties in $Hom(F_2,U(2))$.)
Put $E={\Bbb C}^2$.

Then any $H=H_{w,e} $ with $0\not=e\in E, 1\not=w\in F_2$ is
bounded on any cyclic subgroup, but many are globally unbounded. 
The second statement
follows by noting that the orbit of any unit vector under a nontrivial
cyclic subgroup is dense in a torus $S^1\times S^1\subset \mathbb
C^2$, so (ii) of Lemma \ref{unbounded} fails, and (i) must hold. 
 For the first statement, observe that for a fixed $g\in F$
the
values $H(g^n)$ can be computed, up to a bounded error, by adding sums of
the form
$$U_n=u(e)+gu(e)+\cdots+g^{n-1}u(e)$$
one for every $g$-orbit of occurrences of $w$ or $w^{-1}$ along the axis of
$g$. 
Applying $g$ we have 
$$g(U_n)=gu(e)+\cdots +g^nu(e)$$ and so $|g(U_n)-U_n|\leq 2|e|$, which
implies that $|U_n|$ is bounded, since $g:\mathbb C^2\to\mathbb C^2$
moves every unit vector a definite amount. It follows $H(g^n)$ is bounded on $n$.
This gives an isometric
quasi-action of $F_2$ on $\mathbb C^2$ or $\R^4$ with unbounded
orbits, but with every cyclic subgroup having bounded orbits.

In fact, since $H^1(F_2;\rho)\neq 0$, it follows that there are {\it
  isometric} actions of $F_2$ on $\R^4$ with unbounded orbits and with
every element fixing a point.
\end{example}

The following is our basic method of detecting bounded cocycles. In
the presence of reflexivity of the Banach space, bounded isometric
actions have fixed points. Thus a cocyle $G:F_2\to E$ is bounded if
and only if for some $v\in E$ (a fixed point of the action) we have
$G(g)=v-\rho(g)v$ for every $g\in F_2$.

\begin{lemma}\label{bounded cocycle 2}
Let $\rho$ be a unitary representation of $F_2$ on a reflexive Banach space $E$
and $G$ a cocycle that is bounded on $\langle a^2,b \rangle$ and
$\langle a^3, b \rangle$. Then one of the following holds.
\begin{enumerate}[(i)]
\item $G$ is bounded on $F_2$, or

\item There is a free subgroup $F\subset F_2$ with $F\cong F_2$ such that
  $\rho|_F$ fixes a nonzero vector in $E$.
  \end{enumerate}
  \end{lemma}

\begin{proof}
The cocycle $G$ induces an action of $F_2$ on $E$ by affine isometries
and the image of $G$ is the orbit of $0$ under this action. If the
restriction of this action to $\langle a^2, b \rangle$ is bounded
(with respect to the norm topology) then the convex hull of the orbit
(in the weak topology) will be $\langle a^2, b \rangle$-invariant and
compact since $E$ is reflexive so by the Ryll-Nardzewski fixed point
theorem $\langle a^2,b \rangle$ will have a fixed point. Thus
$Fix(a^2)\cap Fix(b)\neq\emptyset$. If this intersection is not a
single point then (ii) holds since $\rho$ restricted to $F=\langle
a^2,b\rangle$ fixes the difference of any two vectors in the
intersection. Similarly, (ii) holds if $Fix(a^3)\cap
Fix(b)\neq\emptyset$ is not a single point. If the two intersections
coincide then the intersection point is fixed by both
$a=a^3(a^2)^{-1}$ and $b$, so $G$ is bounded. If the intersections are
distinct then $F=\langle a^6,b\rangle$ fixes two distinct points, so
(ii) holds.
\end{proof}

\subsection{Detecting essentiality and proof of Theorem \ref{main:free2}}

We now show that under suitable conditions our quasi-cocycles are
essential. We consider two cases. If there is a free subgroup that
fixes a nonzero vector $e\in E$, the argument essentially goes back to
Brooks, since in this case we restrict to the trivial
representation. This case is presented first.

\begin{prop}\label{fixed vector:free}
Let $\rho$ be a unitary representation of $F_2$ in a reflexive Banach
space $E$ and let $F$ be a rank two free subgroup such that $\rho|_F$
has an invariant vector $e\neq 0$. Let $w'$ be a cyclically reduced
word that is conjugate into $F$. Then quasi-cocycles of the form
$H_{w, e}$ where $w$ is a reduced word that contains $w'$ as a subword
span an infinite dimensional subspace of
$\widetilde{QC}(F_2;\rho)$.
\end{prop}

The word $w'$ can be empty and that is the case we use later. 
\begin{proof}
After possibly conjugating $F$ we can assume that $w'$ is contained in
$F$. Since $w'$ is cyclically reduced its axis contains the identity
in the Cayley graph for $F_2$. This implies that the minimal $F$-tree
also
contains the identity and allows us to find cyclically reduced words
$\alpha$ and $\beta$ in $F$ such that the concatenation
$$w_k = w'\alpha^k \beta^k\alpha^k\beta^k$$ is cyclically
reduced. Furthermore we can assume that $\alpha$ and $\beta$ generate
$F$. Let $H_k=H_{w_k,e}$. By Lemma \ref{fixed points} there exists an
$F$-invariant (continuous) linear functional $\phi$ with $\phi(e) \ge
\mu>0$. 

Then  the restriction to $F$ of the
composition $\phi \circ G$ with any co-cycle $G$ is a homomorphism,
and similarly the restriction of the composition $\phi \circ H$ to
$F$ with any quasi-co-cycle $H$ is a quasi-morphism.

We will
show that the sequence $H_{k_0},H_{k_0+1},\cdots$ represents linearly
independent elements in $\widetilde {QC}(F_2;\rho)$ for $k_0$ so large
that all powers of $\alpha$ and $\beta$ appearing in $w'$ are $<<k_0$
in absolute value. 
% It suffices to show that the restrictions of the 
%quasi-co-cycle $H_i$ to
%$F$ represent linearly independent quasi-morphisms on $F$. 
Indeed, if
$H=H_k-c_{k_0}H_{k_0}-\cdots -c_{k-1}H_{k-1}$, with $k_0 <k$,  for any constants $c_i$
then the quasi-morphism $\phi \circ H$ on $F$ is 0 on the powers of $\alpha$
and $\beta$, so if a co-cycle $G$ is boundedly close $H$, then the homomorphism $\phi
\circ G$ on $F$  must be bounded, and therefore zero, on powers of $\alpha$
and $\beta$. Therefore $\phi \circ G$ is trivial when restricted to
$F$. On the other hand a  staightforward calculation shows that $\phi
\circ H(w^n_k) \ge n \mu$
so $\phi \circ H$ is unbounded on $F$ and $H$ and $G$ cannot be boundedly close.
We showed that $H$ is non-trivial in $\widetilde {QC}(F_2;\rho)$,
so $H_{k_0},H_{k_0+1},\cdots, H_k$ are linearly independent. 

\end{proof}

We now consider the opposite case when no reduction to the trivial
representation is possible.

\begin{prop}\label{no fixed vector:free}
Let $\rho$ be a unitary representation of $F_2= \langle a, b \rangle$
on a uniformly convex Banach space and assume that no nonabelian
subgroup of $F_2$ fixes a nonzero vector. Then for any fixed $e\neq 0$
and a cyclically reduced word $w'$ the quasi-cocycles of the form
$H_{w, e}$ span an infinite dimensional subspace of
$\widetilde{QC}(F_2;\rho)$, where $w$ range over cyclically reduced
words that contain $w'$ as a subword.
\end{prop}
The word $w'$ can be empty, and that is the case we use later. 
\begin{proof}
We assume that $w'$ does not start with
$b^{-1}$ or end with $a^{-1}$ (otherwise, swap $a$ and $b$ in 
the definition of $w_m$). Let $w_m=w'a^{5m}b^{5m}a^{7m}b^{7m}$,
$m\geq 1$, and $gcd(m,6)=1$. By Lemma \ref{unbounded}, $H_m
=H_{w_m,e}$ is unbounded. Furthermore $H_m$ is 0 on the subgroups
$\langle a^2,b\rangle$ and $\langle a^3,b\rangle$
listed in Lemma \ref{bounded cocycle 2}.

We claim that those $H_m$'s are linearly independent in $\widetilde{QC}(F_2;\rho)$. 
Fix $m$ and let
$H=H_m-\sum_{i<m}c_iH_i$ for constants $c_i$. Then $H$ is also
unbounded, since the $H_i$ for $i<m$ are visibly 0 on all words in
$F_{w_m}$, the set given in  the proof of Lemma \ref{unbounded},
 but $H_m$ is unbounded on
$F_{w_m}$.
$H$ is bounded on $\langle a^2,b\rangle$ and $\langle a^3,b\rangle$.

Suppose $H$ differs from a cocycle $G$ by
a bounded function. Then $G$ is also bounded on the subgroups
$\langle a^2,b\rangle$ and $\langle a^3,b\rangle$, therefore
$G$ is bounded on $F_2$ since (i) must holds in Lemma \ref{bounded cocycle 2}.
So, $H$ is bounded on $F_2$, contradiction. 
We showed that $H_i, i\le m$ are linearly independent in $\widetilde{QC}(F_2;\rho)$.
\end{proof}

Theorem \ref{main:free2} now follows immediately.
If there is a rank two free subgroup $F$ in $F_2$
with an $F$-invariant vector $e\not=0$, then use Proposition \ref{fixed vector:free} with $w'$ empty to produce an infinite dimensional subspace. 
Otherwise, use Proposition \ref{no fixed vector:free} with $w'$ empty.

\begin{thm}\label{main:free}
Let $\rho$ be a unitary representation of $F_2$ on a uniformly convex
Banach space $E\neq 0$.
Then
$\dim\widetilde{QC}(F_2;\rho)=\infty$. 
\end{thm}

We remark that Pascal Rolli has a new construction, different from the
Brooks construction, that he showed in \cite{rolli} produces
nontrivial quasi-cocycles on $F_2$ (and some other groups) when the
Banach space $E$ is an $\ell^p$-space (or finite dimensional).

\begin{example}\label{ex.ell.0}
To see the importance of uniform convexity we will look more closely
at the examples $\ell^\infty(F_2)$ of bounded functions and
$\ell^\infty_0(F_2)$ of bounded functions that vanish at infinity. 

(1) 
For
the regular representation on $\ell^\infty(F_2)$ (or any group $G$)
the constant functions determine an one-dimensional invariant
subspace. In particular, any quasi-morphism canonically determines a
quasi-cocycle with image in this invariant subspace. If the original
quasi-morphism is essential one may expect that the associated
quasi-cocycle is also essential. However, for any quasi-cocycle $H$ we
can we define the function $H_0:F_2 \to \ell^\infty(F_2)$ by
$H_0(g)(f) = H(f)(f) -
\rho(g)H(g^{-1}f)(f)=H(f)(f)-H(g^{-1}f)(g^{-1}f)$ and then we can
check that $H_0$ is a cocycle (essentially it is the coboundary of the
0-cochain defined by the function $f \mapsto H(f)(f)$) and that $\|H
- H_0\|_\infty \le \Delta(H)$. In particular, $\widetilde{QC}(F_2;\ell^\infty(F_2))=0$
and $H^2_b(F_2;\ell^\infty(F_2)) = 0$.

(2)
For the regular representation of $F_2$ on $\ell^\infty_0(F_2)$
neither $F_2$ nor any non-trivial subgroup fixes a non-trivial
subspace so we cannot, as in the $\ell^\infty(F_2)$ case, use
quasi-morphisms to construct unbounded quasi-cocycles. Furthermore for
some choices of the vector $e$, the quasi-cocyle $H_{w,e}$ will be
bounded. For example if $e \in \ell^\infty_0(F_2)$ is defined by
$$ e(x) = \left\{ \begin{array}{ll} 1 & x = 1 \\ 0 & x \neq
  1 \end{array} \right.$$ then $\|H_{w,e}(x)\|_\infty = 0 \mbox{ or }
1$ depending on whether $x$ does or doesn't contain a copy of
$w$. More generally if $e \in \ell^1(F_2) \subset \ell^\infty_0(F_2)$
we have that $\|H_{w,e}(x)\|_\infty \le \|e\|_1$.  On the other hand
if we define $f \in \ell^\infty_0(F_2)$ by
$$ f(x) = \left\{ \begin{array}{ll} 1/n & x = w^{-n}, n>0\\ 0 &
  \operatorname{otherwise}\end{array} \right.$$ then $|H_{w,
  f}(w^n)(id)| = \sum_{i=1}^n 1/i$ so $\|H_{w,f}(w^n)\|_\infty$ is
unbounded. We can still construct the cocycle $H_0$ as in the previous
paragraph where $H=H_{w,f}$ but this cocycle will not lie in
$\ell^\infty_0(F_2)$. This example emphasizes an inherent difficulty
in extending our results to a wider class of Banach spaces. 

Note that
$H^n_b(G;\ell^{\infty}(G))=0 (n \ge 1)$ for any group $G$ \cite[Proposition 7.4.1]{monod}
since $\ell^{\infty}(G)$ is a ``relatively injective'' Banach $G$-module
\cite[Chapter II]{monod}, so some
assumption on the Banach space is necessary. 

(3) We also note that $H^2_b(G;\ell_0^{\infty}(G))=0$ for any
countable, {\it exact} group (e.g. $G=F_2$, see \cite{ozawaicm}). This can be seen as
follows.  First, since $\ell^\infty(G)$ is a relatively injective
Banach $G$-module, $H^n_b(G;\ell^{\infty}(G))=0$ for all $n>0$.  From
the long exact sequence in bounded cohomology (\cite[Proposition
  8.2.1]{monod}) induced by the short exact sequence
$0\to\ell_0^\infty(G)\to \ell^\infty(G)\to
\ell^{\infty}(G)/\ell_0^{\infty}(G) \to 0$, it suffices to show
$H^1_b(G, \ell^\infty(G)/\ell_0^{\infty}(G))=0$.
But this holds if $G$ is countable and exact
  \cite[Theorem 3]{ozawa3}.  We thank Narutaka Ozawa for pointing out his
work to us.

To show  $H^n_b(G;\ell_0^\infty(G))=0$ for all $n>1$
it suffices to know $$H^n_b(G, \ell^\infty(G)/\ell_0^{\infty}(G))=0$$
for all $n>0$. Ozawa informs us that this is also true. 

%Note that the quotient
%$\ell^\infty(G)/\ell_0^\infty(G)$ can be identified with the Banach
%space $C(X)$ of continuous functions on the compact space $X=\beta
%G-G$, where $\beta$ denotes the Stone-\v Cech
%compactification. 
%Since $G$ is exact, the action of $G$ on $X$
%is (topologically) amenable, \cite[Section 2]{ozawaicm}.
%Note that if the action of $G$ on $X$ was amenable in the sense
%of Zimmer (see \cite[Section 5]{monod}), then  by \cite{BM} 
%it would follow that the space $C(X)$ is a
%relatively injective Banach $G$-module, so $H^n_b(G;C(X))=0$ for $n>0$.
\end{example}

\section{Hyperbolically embedded subgroups}\label{s:hyp-embed}
Before proving our main theorem we need a couple of straightforward lemmas.

\begin{lemma}\label{finite normal}
Let $\rho$ be a unitary representation of a group $G$ on $E$ and $K$ a
finite normal subgroup. Let $E' \subset E$ be the closed subspace of
$K$-invariant vectors and $\rho'$ the unitary representation of $G$ on
$E'$ obtained by restriction. Then every (quasi)-cocyle in $QC(G;
\rho)$ is a bounded distance from a (quasi)-cocyle in $QC(G;\rho')$
\end{lemma}

\begin{proof}
We first define a  linear projection $\pi: E \to E'$ by
$$\pi(x) = \frac{1}{|K|} \sum_{k \in K} \rho(k) x.$$ If $H$ is a
(quasi)-cocycle in $QC(G; \rho)$ then $\tilde H = \pi \circ H$ is a
(quasi)-cocycle in $QC(G; \rho')$. We need to show that $\tilde H$ is
at bounded distance from $H$.

Recall that $H$ is the translational part of an isometric
$G$(-quasi)-action on $E$. By the normality of $K$ if two points in
$E$ are in the same $G$-orbit then their $K$-orbits are
(quasi)-isometric. Since $H(G)$ is the $G$-orbit of $0$ under this
(quasi)-action and $H(Kg)$ is at bounded distance from
the $K$-orbit of $H(g)$ we have that the
$K$-orbits of points in the image of $H$ are uniformly bounded, and so
$\pi$ moves points in $Im(H)$ a uniformly bounded amount.
\end{proof}

\begin{cor}\label{invariant-extension}
The natural map $\widetilde{QC}(G; \rho') \to \widetilde{QC}(G;
\rho)$ is an isomorphism.
\end{cor}

\begin{lemma}\label{finite-extension}
Let $\rho$ be a unitary representation of $G \times K$ on $E$ such
that $K$ is finite and $\rho$ restricted to the $K$-factor is
trivial. Then there is a natural isomorphism from $\widetilde{QC}(G
\times K; \rho) \to \widetilde{QC}(G; \rho)$.
\end{lemma}

\begin{proof}
Given $H \in QC(G \times K; \rho)$ define $\tilde H \in QC(G; \rho)$
by $\tilde H(g) = H(g, id)$. The linear map defined by $H \mapsto \tilde H$ descends to a
linear map $\widetilde{QC}(G \times K; \rho) \to \widetilde{QC}(G;
\rho)$. Any quasi-cocycle in $QC(G; \rho)$ determines a quasi-cocylce
in $QC(G\times K; \rho)$ by extending it to be constant on the
$K$-factor. This also descends to a map $\widetilde{QC}(G; \rho) \to
\widetilde{QC}(G \times K; \rho)$, which is an inverse of our first
map
since $\|H(g, k) - H(g, id)\| \le
\Delta(H) + C$ where $C = \max\{\|H(id,k)\| | k \in K\}$.
Hence we have the desired isomorphism.
\end{proof}

In \cite{DGO}, Dahmani, Guiradel and Osin defined the notion of a {\em
  hyperbolically embedded subgroup}. For convenience we recount the
definition here. Let $G$ be a group, $H$ a subgroup and $X \subset G$
such that $X \cup H$ generates $G$. Let $\Gamma(G, X \cup H)$ be the
Cayley graph with generating set $X \cup H$. Then $H$ is
hyperbolically embedded in $G$ if
\begin{itemize}
\item $\Gamma(G, X \cup H)$ is hyperbolic;

\item For all $n>0$ and $h \in H$ there are at most finitely many $h'
  \in H$ that can be connected to $h$ in $\Gamma(G,X \cup H)$ by a
  path of length $\le n$ with no edges in $H$.
\end{itemize}
A quasi-cocycle is {\em anti-symmetric} if
$$H(g^{-1}) = -\rho(g^{-1})H(g).$$ A cocycle automatically satisfies
this condition. Furthermore every quasi-cocycle is a bounded distance
from an anti-symmetric quasi-cocycle. (Simply replace $H(g)$ with
$\frac12(H(g) - \rho(g) H(g^{-1})$.) We have the following important
theorem of Hull and Osin.
\begin{thm}[\cite{hull.osin}]\label{Hull-Osin}
Let $G$ be a group and $F$ a hyperbolically embedded subgroup. Then there exists a linear map
$$\iota : QC_{as}(F; \rho) \to QC_{as}(G;\rho)$$ such that if $H \in
QC_{as}(F;\rho)$ then $H = \iota(H)|_F$. In particular, $\dim
\widetilde{QC}(F;\rho) \le \dim \widetilde{QC}(G;\rho)$.
\end{thm}

The action of a group $G$ on a metric space $X$ is {\em acylindrical}
if for all $B>0$ there exist $D,N$ such that if $x,y \in X$ and with
$d(x,y)> D$ then there are at most $N$ elements $g \in G$ with $d(x,
gx)<B$ and $d(y,gy)<B$. A group $G$ is {\em acylindrically hyperbolic}
if it has an acylindrical,, non-elementary, action on a $\delta$-hyperbolic space.  To
apply the previous theorem we need the following result of
Dahmani-Guiradel-Osin and Osin:
\begin{thm}[\cite{DGO, osin2}]\label{hyp-embed}
Let $G$ be an acylindrically hyperbolic group and $K$ the maximal finite normal subgroup. Then $G$ contains a hyperbolically embedded copy of $F_2 \times K$.
\end{thm}

\begin{remark}
Theorem \ref{hyp-embed} is a combination of two theorems. In
\cite[Theorem 1.2]{osin2}, Osin proves that an acylindrically hyperbolical group
contains a non-degenerate hyperbolically embedded  subgroup. In \cite[Theorem 2.24]{DGO},
Dahmani-Guirardel-Osin 
show that if $G$ contains a non-degenerate hyperbolically embedded subgroup then it
contains a hyperbolically embedded copy of $F_2 \times K$. We note
that this latter theorem relies on the projection complex defined in
\cite{bbf}.
\end{remark}

\begin{proof}[Proof of Corollary \ref{main:acylind}]  
Let $E'\subset E$ be the subspace fixed by $K$ and $\rho'$ the
restriction of $\rho$ to $E'$. By assumption $\dim E' > 0$. By Theorem
\ref{hyp-embed} there is a copy of $F_2 \times K$ hyperbolically
embedded in $G$.  By Lemma \ref{finite-extension} and Theorem
\ref{main:free} we have that $\dim \widetilde{QC}(F_2 \times K; \rho')
= \dim \widetilde{QC}(F_2, \rho') = \infty$.  Corollary
\ref{invariant-extension} implies that $\dim \widetilde{QC}(F_2 \times
K; \rho) = \infty$. The corollary then follows from Theorem
\ref{Hull-Osin}. \end{proof}

\bibliography{./../ref2}

\end{document}